\documentclass[12pt]{article}
\usepackage[utf8]{inputenc}
\usepackage[russian]{babel}
\usepackage{amsfonts, amsmath, amsthm}
\usepackage{colortbl}
\usepackage[square, comma, sort&compress, numbers]{natbib}

\usepackage{natbib}
\usepackage{graphicx}

\begin{document}

\newcommand{\widevdots}{\;\vdots\;}
\newcommand{\vdW}{\text{vdW}}
\newcommand{\SR}{\text{SR}}

\newtheorem{opredelenie}{Определение}
\newtheorem{teorema}{Теорема}
\newtheorem*{lemma}{Лемма}
\newtheorem{sledstvie}{Следствие}
\newtheorem*{svojstvo}{Свойство}

\large

\title{Существование одноцветных симплексов заданного объёма на~раскрашенной рациональной решётке}

\author{Шарич Владимир Златкович \and Канель-Белов Алексей (MIPT, BIU)}


\maketitle













\begin{abstract} 
В работе доказывается существование одноцветных стандартных симплексов заданного объема на многомерной рациональной решетке, покрашенной в конечное число цветов.
Данная работа была проведена с помощью Российского Научного Фонда Грант N 17-11-01377.
 \end{abstract}
\tableofcontents



\section{Введение}

В данной работе рассматривается рациональное координатое $n$-мерное пространство, состоящее из всех точек $(x,y,\dots,z)$, где $x,y,\dots,z\in\mathbb{Q}$. Такие точки будем называть {\it рациональными}.

Кроме того, каждой точке рационального пространства плоскости приписан цвет (формально говоря, все рациональные точки разбиты на несколько непересекающихся классов).

\begin{opredelenie}[Одноцветная фигура] Будем говорить, что фигура {\it одноцветна}, если все её точки имеют одинаковые цвета. \end{opredelenie}

Знаменитая теорема ван дер Вардена \cite{vdWbug} утверждает, что при любой раскраске натурального ряда в конечное число цветов найдётся сколь угодно длинная одноцветная арифметическая прогрессия. Хорошо известно, что аналогичный факт верен в любой размерности: {\it пусть на $n$-мерной целочисленной решётке дана фигура $\mathcal{M}$; при любой раскраске этой решётки в конечное число цветов найдется одноцветная фигура $\mathcal{M'}$, получающаяся из фигуры $\mathcal{M}$ гомотетией с натуральным коэффициентом или сдвигом на вектор с натуральными координатами.}

Гомотетии с натуральными коэффицентами порождают группу преобразований, состоящую из гомотетий с рациональными коэффициентами и сдвигов на векторы с рациональными координатами: $$G_1 = \{\mathcal{H}_P^k, P\in\mathbb{Q}^n, k\in\mathbb{Q},k>0\} \cup \{T_{\overrightarrow{v}}, \overrightarrow{v}\in\mathbb{Q}^n\}.$$ Эта группа действует на множестве фигур рационального координатного пространства $2^{\mathbb{Q}^n}$. Таким образом, теорема ван дер Вардена описывает комбинаторные свойства действия $G_1:2^{\mathbb{Q}^n}$, утверждая, что при любой раскраске пространства в конечное число цветов любая орбита содержит одноцветную фигуру.

Другая интересное действие --- группа изометрий, действующая на $2^{\mathbb{R}^n}$ --- таким свойством не обладает. Известно, что в каждой размерности найдется такое число цветов (например, $h>(4\lfloor\sqrt{n}+1\rfloor)^n$), что в орбите двухэлементного множества может не оказаться одноцветной фигуры \cite{raigorhrom}.

В данной работе мы будем рассматривать положительные рациональные гиперолические повороты плоскости $$\mathcal{U}_{(x_0,y_0)}^k\colon (x,y)\mapsto \left(x_0+k(x-x_0),y_0+\dfrac{y-y_0}{k}\right), \qquad x_0,y_0\in\mathbb{Q},k\in\mathbb{Q},k>0.$$ Нас будет интересовать действие на $2^{\mathbb{Q}^2}$ группы $G_2$, порождённой положительными рациональными гиперболическими поворотами. Эта группа состоит из положительных рациональных гиперболических поворотов и сдвигов на векторы с рациональными координатами: $$G_2=\{\mathcal{U}_P^k, P\in\mathbb{Q}^2, k\in\mathbb{Q},k>0\} \cup \{T_{\overrightarrow{v}}, \overrightarrow{v}\in\mathbb{Q}^2\}.$$

Для последующих теорем нам понадобится обозначение $\vdW_h(N)$ (где $N$ --- число) из следующей формулировки-усиления теоремы ван дер Вардена: <<{\it Найдётся такое число $\vdW_h(N)\in\mathbb{N}$, что из любой арифметической прогрессии длины $\vdW_h(N)$, покрашенной в $h$ цветов, можно выбрать одноцветную подпрогрессию длины $N$.}>> В качестве $\vdW_h(N)$ принято выбирать наименьшее число, удовлетворяющее указанным условиям, но это не обязательно. Введём также аналогичное обозначение $\vdW_h(\mathcal{M})$ (где $\mathcal{M}$ --- фигура) из многомерного обобщения усиления теоремы ван дер Вардена: <<{\it Пусть на $n$-мерной целочисленной решётке дана фигура $\mathcal{M}$; найдётся число $\vdW_h(\mathcal{M})\in\mathbb{N}$, такое что при любой раскраске $n$-мерного куба $$\underbrace{\vdW_h(\mathcal{M})\times \vdW_h(\mathcal{M}) \times \ldots \times \vdW_h(\mathcal{M})}_{n}$$ в $h$ цветов в этом кубе найдется одноцветная фигура $\mathcal{M'}$, получающаяся из фигуры $\mathcal{M}$ композицией гомотетии с натуральным коэффициентом и сдвига на вектор с натуральными координатами.}>>

\begin{opredelenie}[Стандартный треугольник] {\it Стандартные треугольники} --- это треугольники вида $$(x_0,y_0), \quad (x_0+b,y_0), \quad (x_0,y_0+a),$$ где $x_0,y_0,a,b\in\mathbb{Q}$, $a,b>0$. \end{opredelenie}

Нас будут преимущественно интересовать стандартные треугольники площади 1/2, т.е. такие, для которых $ab=1$. Все такие треугольники переходят друг в друга при действии $G_2$, т.е. являются одной орбитой.

{\bf Основной двумерный результат:} при любой раскраске рациональной плоскости в конечное число цветов найдется одноцветный стандартный треугольник заданной площади.

\begin{opredelenie}[Стандартный симплекс] {\it Стандартный симплекс} --- это симплекс вида $$(x_0,y_0,\dots,z_0), \quad (x_0+b,y_0,\dots,z_0), \quad (x_0,y_0+a,\dots,z_0),\quad \dots, \quad (x_0,y_0,\dots,z_0+c)$$ где $x_0,y_0,\dots,z_0,a,b,\dots,c\in\mathbb{Q}$, $a,b,\dots,c>0$.  \end{opredelenie}

Нас будут преимущественно интересовать стандартные симплексы объема $1/n$, т.е. такие, для которых $a\cdot b\cdot \ldots \cdot c=1$.

{\bf Основной многомерный результат:} при любой раскраске рационального $n$-мерного пространства в конечное число цветов найдется одноцветный стандартный симплекс заданного объема.

\bigskip



Структура работы такова: \par \begin{itemize}
    \item в разделе <<Косой дождь и его свойства>> описываются инструменты для исследования стандартных треугольников площади 1/2;
    \item в разделе <<Одноцветные треугольники>> доказывается существование одноцветных стандартных треугольников площади 1/2 (а также любой заданной площади) при раскраске рациональной плоскости в конечное число цветов;
    \item в разделе <<Одноцветные симплексы>> доказывается существование одноцветных стандартных симплексов объема $1/n$ при раскраске рационального пространства в конечное число цветов;
    \item в разделе <<Перспективы>> описываются дальнейшие направления исследования.
\end{itemize}



Данная работа была проведена с помощью Российского Научного Фонда Грант N 17-11-01377.

\section{Косой дождь и его свойства}

Косой дождь --- конструкция, позволяющая отслеживать положение стандартных треугольников площади 1/2.

\medskip

\begin{opredelenie}[Косой дождь] {\it Косой дождь} длины $\ell$ и с шагом $b$ --- набор рациональных точек состоящий из: 
\begin{itemize} 
\item {\it основания} --- горизонтальной арифметической прогрессии длины $\ell$ и шагом $b$;
 \item {\it слоёв} --- $\ell-1$ горизонтальных арифметических прогрессий с шагом $b$, начало которых расположено над началом основания и
  
  \begin{itemize} 
  \item первая из которых имеет длину $\ell-1$ и расположена на высоте $\dfrac{1}{b}$ над основанием,
      \item вторая --- имеет длину $\ell-2$ и расположена над основанием на высоте $\dfrac{1}{2b}$,
       \item \dots,
       \item $k$-я имеет длину $\ell-k$ и расположена над основанием на высоте $\dfrac{1}{kb}$, \item \dots,
       \item $(\ell-1)$-я имеет длину 1 ($=\ell-(\ell-1)$) и расположена над основанием на высоте $\dfrac{1}{\ell b}$.
       \end{itemize}
       \end{itemize}
       
       \par Иными словами, косой дождь --- это $$ \underbrace{\left\{(x_0+\kappa b,y_0),\kappa\in\{0,1,\dots,\ell-1\}\right\}}_{\text{\large \it основание}}$$ $$\cup\underbrace{\left\{\left(x_0+\kappa b,y_0+\dfrac{1}{k b}\right),k\in\{1,2,\dots,\ell-1\},\kappa\in\{0,1,\dots,\ell-k\}\right\}}_{\text{\large \it слои}}.$$ \end{opredelenie}



\medskip

Заметим, что косой дождь определяется своим основанием, а точки в слоях косого дождя --- это в точности всевозможные верхние вершины стандартных треугольников площади 1/2, нижние две вершины которых лежат в основнии косого дождя. Отсюда очевидно следует

\begin{svojstvo} Если основание одного косого дождя является подмножеством основания другого косого дождя, то и слои первого косого дождя содержатся в слоях второго косого дождя. \end{svojstvo}

\bigskip

Нас будет интересовать, содержится ли в слоях достаточно большого косого дождя другой косой дождь наперёд заданной длины. Оказывается, да.

\bigskip

\begin{teorema} \label{poddozhd} Для любого $\ell$ в слоях косого дождя длины $\ell!\ell+1$ содержится (т.е. является подмножеством) косой дождь длины $\ell$. \end{teorema}

\begin{proof} Будем рассматривать дождь длины $L$ и называть его {\it большим}. Искомый дождь длины $\ell$ (существование которого предстоит установить) будем называть {\it малым}. Рассмотрим, какими свойствами должен обладать малый дождь для того, чтобы оказаться подмножеством слоёв большого дождя.

Не умаляя общности, для удобства изложения можно считать, что шаг большого дождя равен 1. Если это не так, применим аффинное преобразование $$(x,y)\mapsto \left(tx,\dfrac{1}{t}y\right)$$ с подходящим $t>0$. Очевидно, что такое преобразование переводит косой дождь в косой дождь.

Пусть малый дождь имеет шаг $b\in\mathbb{N}$, а его основание находится на высоте $\frac{1}{m}$ ($m\in\mathbb{N},m\le L-1$) над основанием большого дождя. Для того, чтобы хотя бы один такой малый дождь содержался в большом, необходимо и достаточно выполнение следующих условий: \begin{gather*} (\ell-1)b+1\leq L-m, \qquad (\ast) \\ \forall k\in{1,2,\dots,\ell-1} \quad \exists n\in\mathbb{N}, n\leq L-1 \quad \dfrac{1}{m}+\dfrac{1}{kb}=\dfrac{1}{n}. \qquad (\ast\ast) \end{gather*}

Заметим, что условие $n\le L-1$ лишнее: оно автоматически следует из того, что $\dfrac{1}{n}>\dfrac{1}{m}$ (откуда $n<m$) и $m\le L-1$.

\medskip

\begin{lemma} Для натуральных чисел $m$ и $kb$ найдется натуральное число $n$ такое, что $$\dfrac{1}{m}+\dfrac{1}{kb}=\dfrac{1}{n},$$ тогда и только тогда, когда $k^2b^2\widevdots m+kb$. \end{lemma}

\begin{proof} \begin{gather*} \exists n\in\mathbb{N} \quad \dfrac{1}{n}=\dfrac{1}{m}+\dfrac{1}{kb}=\dfrac{m+kb}{m\cdot kb} \qquad \Longleftrightarrow \qquad m\cdot kb\widevdots m+kb \\ m\cdot kb\widevdots m+kb \quad\Leftrightarrow\quad kb(m+kb)-m\cdot kb\widevdots m+kb \quad\Leftrightarrow\quad k^2b^2\widevdots m+kb.\end{gather*} \end{proof}

\medskip

Положим $m=b$. Условие $k^2b^2\widevdots m+kb$ преобразуется в условие $k^2b^2\widevdots (k+1)b$, что равносильно $b\widevdots k+1$. Следовательно, для выполнения $(\ast\ast)$ достаточно положить $b=\ell!$.

Для выполнения $(\ast)$ достаточно положить $L=(\ell-1)\ell!+1+\ell!=\ell!\ell+1$.

\end{proof}

\medskip

\centerline{\small\it Иллюстрация теоремы 1: косой дождь длины 19, в слоях которого содержится косой дождь длины 3.}

\nopagebreak

\unitlength=1mm

\hspace{-5mm}
\begin{picture}(180,125)

\put(0,0){\circle*{1}}
\put(10,0){\circle*{1}}
\put(20,0){\circle*{1}}
\put(30,0){\circle*{1}}
\put(40,0){\circle*{1}}
\put(50,0){\circle*{1}}
\put(60,0){\circle*{1}}
\put(70,0){\circle*{1}}
\put(80,0){\circle*{1}}
\put(90,0){\circle*{1}}
\put(100,0){\circle*{1}}
\put(110,0){\circle*{1}}
\put(120,0){\circle*{1}}
\put(130,0){\circle*{1}}
\put(140,0){\circle*{1}}
\put(150,0){\circle*{1}}
\put(160,0){\circle*{1}}
\put(170,0){\circle*{1}}
\put(180,0){\circle*{1}}

\put(0,120){\circle*{0.7}}
\put(10,120){\circle*{0.7}}
\put(20,120){\circle*{0.7}}
\put(30,120){\circle*{0.7}}
\put(40,120){\circle*{0.7}}
\put(50,120){\circle*{0.7}}
\put(60,120){\circle*{0.7}}
\put(70,120){\circle*{0.7}}
\put(80,120){\circle*{0.7}}
\put(90,120){\circle*{0.7}}
\put(100,120){\circle*{0.7}}
\put(110,120){\circle*{0.7}}
\put(120,120){\circle*{0.7}}
\put(130,120){\circle*{0.7}}
\put(140,120){\circle*{0.7}}
\put(150,120){\circle*{0.7}}
\put(160,120){\circle*{0.7}}
\put(170,120){\circle*{0.7}}

\put(0,60){\circle*{0.7}}
\put(10,60){\circle*{0.7}}
\put(20,60){\circle*{0.7}}
\put(30,60){\circle*{0.7}}
\put(40,60){\circle*{0.7}}
\put(50,60){\circle*{0.7}}
\put(60,60){\circle*{0.7}}
\put(70,60){\circle*{0.7}}
\put(80,60){\circle*{0.7}}
\put(90,60){\circle*{0.7}}
\put(100,60){\circle*{0.7}}
\put(110,60){\circle*{0.7}}
\put(120,60){\circle*{0.7}}
\put(130,60){\circle*{0.7}}
\put(140,60){\circle*{0.7}}
\put(150,60){\circle*{0.7}}
\put(160,60){\circle*{0.7}}

\put(0,40){\circle*{0.7}}
\put(10,40){\circle*{0.7}}
\put(20,40){\circle*{0.7}}
\put(30,40){\circle*{0.7}}
\put(40,40){\circle*{0.7}}
\put(50,40){\circle*{0.7}}
\put(60,40){\circle*{0.7}}
\put(70,40){\circle*{0.7}}
\put(80,40){\circle*{0.7}}
\put(90,40){\circle*{0.7}}
\put(100,40){\circle*{0.7}}
\put(110,40){\circle*{0.7}}
\put(120,40){\circle*{0.7}}
\put(130,40){\circle*{0.7}}
\put(140,40){\circle*{0.7}}
\put(150,40){\circle*{0.7}}

\put(0,30){\circle*{0.7}}
\put(10,30){\circle*{0.7}}
\put(20,30){\circle*{0.7}}
\put(30,30){\circle*{0.7}}
\put(40,30){\circle*{0.7}}
\put(50,30){\circle*{0.7}}
\put(60,30){\circle*{0.7}}
\put(70,30){\circle*{0.7}}
\put(80,30){\circle*{0.7}}
\put(90,30){\circle*{0.7}}
\put(100,30){\circle*{0.7}}
\put(110,30){\circle*{0.7}}
\put(120,30){\circle*{0.7}}
\put(130,30){\circle*{0.7}}
\put(140,30){\circle*{0.7}}

\put(0,24){\circle*{0.7}}
\put(10,24){\circle*{0.7}}
\put(20,24){\circle*{0.7}}
\put(30,24){\circle*{0.7}}
\put(40,24){\circle*{0.7}}
\put(50,24){\circle*{0.7}}
\put(60,24){\circle*{0.7}}
\put(70,24){\circle*{0.7}}
\put(80,24){\circle*{0.7}}
\put(90,24){\circle*{0.7}}
\put(100,24){\circle*{0.7}}
\put(110,24){\circle*{0.7}}
\put(120,24){\circle*{0.7}}
\put(130,24){\circle*{0.7}}

\put(0,20){\circle*{0.7}}
\put(10,20){\circle*{0.7}}
\put(20,20){\circle*{0.7}}
\put(30,20){\circle*{0.7}}
\put(40,20){\circle*{0.7}}
\put(50,20){\circle*{0.7}}
\put(60,20){\circle*{0.7}}
\put(70,20){\circle*{0.7}}
\put(80,20){\circle*{0.7}}
\put(90,20){\circle*{0.7}}
\put(100,20){\circle*{0.7}}
\put(110,20){\circle*{0.7}}
\put(120,20){\circle*{0.7}}

\put(0,17){\circle*{0.7}}
\put(10,17){\circle*{0.7}}
\put(20,17){\circle*{0.7}}
\put(30,17){\circle*{0.7}}
\put(40,17){\circle*{0.7}}
\put(50,17){\circle*{0.7}}
\put(60,17){\circle*{0.7}}
\put(70,17){\circle*{0.7}}
\put(80,17){\circle*{0.7}}
\put(90,17){\circle*{0.7}}
\put(100,17){\circle*{0.7}}
\put(110,17){\circle*{0.7}}

\put(0,15){\circle*{0.7}}
\put(10,15){\circle*{0.7}}
\put(20,15){\circle*{0.7}}
\put(30,15){\circle*{0.7}}
\put(40,15){\circle*{0.7}}
\put(50,15){\circle*{0.7}}
\put(60,15){\circle*{0.7}}
\put(70,15){\circle*{0.7}}
\put(80,15){\circle*{0.7}}
\put(90,15){\circle*{0.7}}
\put(100,15){\circle*{0.7}}

\put(0,13){\circle*{0.7}}
\put(10,13){\circle*{0.7}}
\put(20,13){\circle*{0.7}}
\put(30,13){\circle*{0.7}}
\put(40,13){\circle*{0.7}}
\put(50,13){\circle*{0.7}}
\put(60,13){\circle*{0.7}}
\put(70,13){\circle*{0.7}}
\put(80,13){\circle*{0.7}}
\put(90,13){\circle*{0.7}}

\put(0,12){\circle*{0.7}}
\put(10,12){\circle*{0.7}}
\put(20,12){\circle*{0.7}}
\put(30,12){\circle*{0.7}}
\put(40,12){\circle*{0.7}}
\put(50,12){\circle*{0.7}}
\put(60,12){\circle*{0.7}}
\put(70,12){\circle*{0.7}}
\put(80,12){\circle*{0.7}}

\put(0,11){\circle*{0.7}}
\put(10,11){\circle*{0.7}}
\put(20,11){\circle*{0.7}}
\put(30,11){\circle*{0.7}}
\put(40,11){\circle*{0.7}}
\put(50,11){\circle*{0.7}}
\put(60,11){\circle*{0.7}}
\put(70,11){\circle*{0.7}}

\put(0,10){\circle*{0.7}}
\put(10,10){\circle*{0.7}}
\put(20,10){\circle*{0.7}}
\put(30,10){\circle*{0.7}}
\put(40,10){\circle*{0.7}}
\put(50,10){\circle*{0.7}}
\put(60,10){\circle*{0.7}}

\put(0,9){\circle*{0.7}}
\put(10,9){\circle*{0.7}}
\put(20,9){\circle*{0.7}}
\put(30,9){\circle*{0.7}}
\put(40,9){\circle*{0.7}}
\put(50,9){\circle*{0.7}}

\put(0,8){\circle*{0.7}}
\put(10,8){\circle*{0.7}}
\put(20,8){\circle*{0.7}}
\put(30,8){\circle*{0.7}}
\put(40,8){\circle*{0.7}}

\put(0,7){\circle*{0.7}}
\put(10,7){\circle*{0.7}}
\put(20,7){\circle*{0.7}}
\put(30,7){\circle*{0.7}}

\put(0,6){\circle*{0.7}}
\put(10,6){\circle*{0.7}}
\put(20,6){\circle*{0.7}}

\put(0,5){\circle*{0.7}}
\put(10,5){\circle*{0.7}}

\put(0,20){\circle{2}}
\put(60,20){\circle{2}}
\put(120,20){\circle{2}}

\put(0,40){\circle{2}}
\put(60,40){\circle{2}}

\put(0,30){\circle{2}}

\end{picture}

\medskip

\begin{sledstvie} В слоях достаточно длинного косого дождя содержится косой дождь любой наперёд заданной длины. \end{sledstvie}

Для дальнейшего удобства введем функцию $f(x)=x!x+1$.


\newpage

\section{Одноцветные треугольники}

Напомним, что {\it стандартный треугольник} --- это треугольники вида $$(x_0,y_0), \quad (x_0+b,y_0), \quad (x_0,y_0+a),$$ где $x_0,y_0,a,b\in\mathbb{Q}$, $a,b>0$. Вершину $(x_0,y_0+a)$ будем называть {\it верхней}, оставшиеся две вершины --- {\it нижними}.

\bigskip

\begin{teorema} \label{treug2cveta} Если все рациональные точки косого дождя длины $\vdW_2(5)$ покрашены в 2 цвета, то в этом косом дожде найдётся стандартный треугольник площади $1/2$, все вершины которого одного цвета. \end{teorema}

\begin{proof}

Будем называть цвета {\it первым} и {\it вторым}.

Рассмотрим одноцветную горизонтальную арифметическую прогрессию из 5 членов. Существование такой прогрессии обеспечивается теоремой Ван дер Вардена. Пусть шаг этой прогрессии равен $b$, а цвет --- первый.

Рассмотрим косой дождь, основанием которого является данная прогрессия. Согласно Свойству 0, он является поддождём исходного дождя.

Слой на высоте $\dfrac{1}{b}$ (назовём его {\it верхним}) состоит из 4 точек. Если хотя бы одна из них имеет первый цвет, треугольник найден. Пусть они все второго цвета.

Слой на высоте $\dfrac{1}{2b}$ (назовём его {\it средним}) состоит из 3 точек. Если хотя бы одна из них имеет первый цвет, треугольник найден. Пусть они все второго цвета.

Легко убедиться, что в таком случае одна точка из вехнего слоя и две точки из среднего слоя образуют стандартный треугольник площади 1/2, все вершины которого второго цвета.

\end{proof}

\begin{sledstvie} Если все рациональные точки плоскости окрашены в 2 цвета, то найдётся стандартный треугольник площади $1/2$, все вершины которого одного цвета. \end{sledstvie}

\begin{teorema}
Если все рациональные точки косого дождя длины $\vdW_3(f(\vdW_2(5)))$ окрашены в 3 цвета, то найдётся одноцветный стандартный треугольник площади $1/2$.
\end{teorema}

\begin{proof}

Рассмотрим в основании данного косого дождя одноцветную арифметическую прогрессию из $f(\vdW_2(5))$ членов. Её существование обеспечивается теоремой Ван дер Вардена.

Рассмотрим косой дождь, построенный на этой прогрессии как на основании. Если хотя бы одна точка в слоях этого косого дождя имеет тот же цвет, что и точки в основании, --- треугольник найден. Пусть все точки в слоях покрашены в один из двух оставшихся цветов.

Согласно теореме \ref{poddozhd}, в слоях найдется косой дождь длины $\vdW_2(5)$. Напомним, что все точки этого косого дождя покрашены в один из двух цветов. Согласно Теореме 2, в таком косом дожде найдется стандартный треугольник площади 1/2.

\end{proof}

\begin{sledstvie}
Если все рациональные точки плоскости окрашены в 3 цвета, то найдётся стандартный треугольник площади $1/2$, все вершины которого одного цвета. \end{sledstvie}

Обобщением теоремы \ref{treug2cveta} является основное утверждение работы. Для его доказательства сформируем последовательность $\SR(h)$ рекуррентным образом: $$\SR(1)=2, \qquad \SR(2)=\vdW_2(5), \qquad \SR(3)=\vdW_3(f(\SR(2))),$$ $$\SR(h)=\vdW_h(f(\SR(h-1))) \quad (h\ge 3).$$

{\bf Примечание.} {\it SR = <<slanting rain>>.}

\begin{teorema} \label{treugMNOGOcvetov}
Если все точки косого дождя длины $\SR(h)$ покрашены в $h$ цветов, то в этом косом дожде содержится одноцветный стандартный треугольник площади 1/2.
\end{teorema}

\begin{proof}

Доказательство будем проводить индукцией по $h$.

База индукции ($h=2$) содержится в Теореме 2.

Предположим, что для $h-1$ цветов утверждение доказано, и произведём переход индукции.

Рассмотрим произвольный косой дождь длины $\SR(h)=\vdW_h(f(\SR(h-1)))$. Согласно теореме Ван дер Вардена, из основания этого дождя можно выделить одноцветную прогрессию длины $f(\SR(h-1))$.

Рассмотрим косой дождь, для которого эта прогрессия является основанием. Если хотя бы одна из точек в слоях этого дождя имеет тот же цвет, что и точки в основании, --- искомый треугольник найден. Предположим, что все точки в слоях этого дождя покрашены в один из оставшихся $h-1$ цветов.

Согласно теореме \ref{poddozhd}, в слоях исходного дождя найдётся поддождь длины $\SR(h-1)$. Согласно предположению индукции, в таком поддожде найдётся искомый треугольник.

\end{proof}

\begin{sledstvie}
Если все рациональные точки плоскости окрашены в конечное число цветов, то найдётся одноцветный стандартный треугольник площади $1/2$.
\end{sledstvie}

\begin{sledstvie}
Если все рациональные точки плоскости окрашены в конечное число цветов, то найдётся одноцветный стандартный треугольник заданной положительной площади $S\in\mathbb{Q}$.
\end{sledstvie}

\begin{proof}

Пусть плоскость покрашена в $h$ цветов. Аффинное преобразование $$(x,y)\mapsto \left(\dfrac{1}{2S}x,y\right)$$ задает новую раскраску плоскости в $h$ цветов. В новой раскраске существует стандартный треугольник площади 1/2. Его прообразом был стандартный треугольник площади $S$.

\end{proof}


\newpage

\section{Одноцветные симплексы}

В многомерном случае мы рассмотрим следующую группу преобразований пространства:
$$G_3=\left\{\mathcal{U}_P^{k_x,k_y,\dots,k_z}, P\in\mathbb{Q}^n, \begin{array}{c} k_x,k_y,\dots,k_z\in\mathbb{Q}, \\ k_x,k_y,\dots,k_z>0, \\ k_x\cdot k_y\cdot \dots\cdot k_z=1 \end{array} \right\} \cup \{T_{\overrightarrow{v}}, \overrightarrow{v}\in\mathbb{Q}^n\},$$ где $$\mathcal{U}_{(x_0,y_0,\dots,z_0)}^{k_x,k_y,\dots,k_z} \colon (x,y,\dots,z)\mapsto \left(x_0+k_x(x-x_0),y_0+k_y(y-y_0),\dots,z_0+k_z(z-z_0)\right).$$

Напомним, что {\it стандартный симплекс} --- это симплекс вида $$(x_0,y_0,\dots,z_0), \quad (x_0+a,y_0,\dots,z_0), \quad (x_0,y_0+b,\dots,z_0),\quad \dots, \quad (x_0,y_0,\dots,z_0+c),$$ где $x_0,y_0,\dots,z_0,a,b,\dots,c\in\mathbb{Q}$, $a,b,\dots,c>0$. Грань этого симплекса, лежащую в гиперплоскости $x=x_0$, будем называть его {\it основанием}, точку $(x_0+a,y_0,\dots,z_0)$ --- {\it вершиной}, а число $a$ --- {\it высотой}.

\bigskip

Действие группы $G_3$ на $2^{\mathbb{Q}^n}$ обладает свойством, аналогичным действию $G_2$ на $2^{\mathbb{Q}^2}$: мы увидим, что орбита любого стандартного симплекса под действием этой группы содержит одноцветную фигуру. Заметим, что два стандартных симплекса принадлежат одной орбите в том и только в том случае, когда их объемы равны.

\begin{opredelenie}[Многомерный косой дождь] {\it $n$-мерный косой дождь длины $\ell\in\mathbb{N}$ с шагом $(s_y,\dots,s_z)\in\mathbb{Q}^{n-1}$} и началом в $(x_0,y_0,\dots,z_0)\in\mathbb{Q}^n$ --- это набор рациональных точек, состоящий из: \par \begin{itemize} \item {\it основания} --- $(n-1)$-мерной решётки $$(x_0,y_0,\dots,z_0)+\{0\} \times \{0,s_y,2s_y,\dots,(\ell-1)s_y\} \times \ldots \times \{0,s_z,2s_z,\dots,(\ell-1)s_z\},$$ лежащей в аффинной гиперплоскости $x=x_0$; \item {\it слоёв} --- всех вершин стандартных симплексов объема $1/n$, основание которых лежит в основании данного многомерного косого дождя; один слой --- это множество $$\hspace{-2em} \left\{\left(x_0+\dfrac{1}{w\cdot s_y\cdot \ldots \cdot s_z},y_0+\kappa_ys_y,\dots,z_0+\kappa_zs_z\right), \begin{array}{c} w\in\{1,2,\dots,(\ell-1)^{n-1}\}, \\ 0\le\kappa_y,\dots,\kappa_z\le\ell-1, \\ (\ell-1-\kappa_y)\cdot\ldots\cdot(\ell-1-\kappa_z)\ge w \end{array} \right\};$$ число $\dfrac{1}{w\cdot s_y\cdot \ldots \cdot s_z}$ будем называть {\it высотой} соответствующего слоя. \end{itemize} \end{opredelenie}

Следующая теорема устанавливает, что в слоях достаточно большого многомерного косого дождя содержится многомерный косой дождь заданной длины. Её доказательство полностью аналогично двумерному случаю.

Введём обозначение $F_n(t)=((t-1)^{n-1}+1)!\cdot t+1.$

\begin{teorema} \label{mnogopoddozhd} Всякий $n$-мерный косой дождь длины $L=F_n(\ell)$ содержит в своих слоях $n$-мерный косой дождь длины $\ell$. \end{teorema}

\begin{proof}

Назовем дождь длины $L$ {\it большим}. Пусть его начало --- точка $(x_0,y_0,\dots,z_0)\in\mathbb{Q}^n$.

Не умаляя общности, для удобства изложения можно считать, что шаг большого дождя равен $(1,\dots,1)\in\mathbb{Q}^{n-1}$, а его слои находятся на высоте $1/W$, где $W$ пробегает все целые значения из диапазона $[1,(L-1)^{n-1}]$. Если это не так, применим аффинное преобразование $$(x,y,\dots,z)\mapsto \left(t_xx,t_yy,\dots,t_zz\right)$$ с подходящими  $t_x,t_y,\dots,t_z>0$, $t_xt_y\dots t_z=1$. Очевидно, что такое преобразование переводит многомерный косой дождь в многомерный косой дождь.

Положим $s=((\ell-1)^{n-1}+1)!$.

Рассмотрим слой большого дождя, расположенный на высоте $a=\dfrac{1}{s^{n-1}}$. В этом слое содержится $(n-1)$-мерная решетка $$(x_0+a,y_0,\dots ,z_0)+\{0\}\times\{0,s,2s,\dots ,(\ell-1)s\}^{n-1},$$ поскольку $$(L-1-\kappa_ys)\cdot\ldots\cdot(L-1-\kappa_zs) \ge (L-1-(\ell-1)s)^{n-1} \ge s^{n-1}\qquad \forall \kappa_y,\dots ,\kappa_z\in\{0,1,\dots ,\ell-1\}\;.$$ Назовем многомерный косой дождь с шагом $(s,\dots ,s)\in\mathbb{Q}^{n-1}$, основанием которого является эта решётка, {\it малым}. Докажем, что весь малый дождь содержится в большом дожде.

Слои малого дождя расположены на высоте $\dfrac{1}{ws^{n-1}}$ над его основанием, причём $w\in\left\{1,2,\dots ,(\ell-1)^{n-1}\right\}$. Осталось заметит, что  $$a+\dfrac{1}{ws^{n-1}}=\dfrac{1+w}{ws^{n-1}}=\dfrac{1}{w\cdot s^{n-2}\cdot 1\cdot \ldots \cdot \widehat{(w+1)}\cdot \ldots \cdot ((\ell-1)^{n-1}+1)}=\dfrac{1}{W};$$ очевидно, знаменатель последней дроби находится в диапазоне $[1,(L-1)^{n-1}]$.

\end{proof}

Обозначим $\mathcal{M}_d(\ell)=\{0,1,\dots ,\ell-1\}^d$ --- $d$-мерный куб из $\ell^d$ точек. Определим последовательность $\SR_n(h)$ рекуррентным образом: $$\SR_n(1)=2, \qquad \SR_n(h)=\vdW_h(\mathcal{M}_{n-1}(F_n(\SR_n(h-1)))) \quad (h\ge 2).$$

\begin{teorema} \label{simplex} В любом $n$-мерном косом дожде длины $SR_n(h)$, покрашенном в $h$ цветов, cуществует одноцветный стандартный симплекс объема $1/n$. \end{teorema}

\begin{proof}

Доказательство индукцией по числу цветов $h$.

{\it База индукции.} Пусть $h=1$. Тогда любая фигура одноцветна.

{\it Переход индукции.} Пусть $h>1$ и пусть утверждение доказано для любого числа цветов, меньшего $h$. Рассмотрим $n$-мерный косой дождь длины $SR_n$. В силу многомерной теоремы ван дер Вардена в его основании можно выбрать одноцветный $(n-1)$-мерный куб $\mathcal{M}$ со стороной $L=F_n(\SR_n(h-1))$. Рассмотрим $n$-мерный косой дождь, основанием которого служит $\mathcal{M}$. Если хотя бы одна из точек в слоях этого дождя покрашена в тот же цвет, что и основание, симплекс обнаружен. Если все точки в слоях этого дождя покрашены в цвета, отличные от цвета основания, то рассмотрим $n$-мерный косой дождь длины $\SR_n(h-1)$, являющийся подмножеством слоёв исходного. Такой дождь существует согласно теореме \ref{mnogopoddozhd}. По предположению индукции в этом поддожде содержится одноцветный стандартный симплекс. \end{proof}

\begin{sledstvie} При любой раскраске $\mathbb{Q}^n$ в конечное число цветов найдется бесконечно много одноцветных стандартных симплексов одинакового цвета и заданного положительного рационального объема. \end{sledstvie}

\begin{proof} Пусть требуемый объем $V\in\mathbb{Q}$, а число цветов $h\in\mathbb{N}$. Аффинное преобразование $$(x,y,\dots ,z)\mapsto \left(\dfrac{1}{nV}x,y,\dots ,z\right)$$ задает новую раскраску пространства. Согласно теореме \ref{simplex} в новой раскраске в каждом $n$-мерном косом дожде со стороной $\SR_n(h)$ существует стандартный симплекс объема $1/n$. Следовательно, таких симплексов бесконечно много, а значит, можно выбрать бесконечно много симплексов одинакового цвета. Их прообразы являются искомыми симплексами. \end{proof}

{\bf Примечание.} Теорема \ref{treugMNOGOcvetov} является частным (двумерным) случаем теоремы \ref{simplex}, равно как и понятие и свойства косого дождя на плоскости.


\newpage

\section{Перспективы}

\subsection*{Более сложные фигуры}

Интересно было бы установить существование не только одноцветных $n$-мерных стандартных симплексов заданного объема, но и одноцветных стандартных параллелепипедов заданного объема, т.е. параллелепипедов $$\begin{array}{c}(x_0+\varepsilon_xa,y_0+\varepsilon_yb,\dots ,z_0+\varepsilon_zc), \\  \varepsilon_x,\varepsilon_y,\dots ,\varepsilon_z\in\{0,1\},\quad a,b,\dots ,c>0, \quad a\cdot b\cdot \ldots \cdot c=V,\end{array}$$ а также одноцветных комбинаторно более сложных структур заданного объема, либо построить (или иным способом доказать существование) раскрасок, в которых эти структуры отсутствуют.

\subsection*{Более сложные группы}

Интерес представляет группа $$G_4 = \left\{\mathcal{U}_P^{k^{n-1}, \underbrace{\tfrac{1}{k},\dots ,\tfrac{1}{k}}_{n-1}}, k\in\mathbb{Q},k>0\right\} \cup \{T_{\overrightarrow{v}}, \overrightarrow{v}\in\mathbb{Q}^n\},$$ (а также иные группы, в которых преобразования $\mathcal{U}$ зависят от одного лишь параметра $k$ (как и в теореме ван дер Вардена). В качестве характеристического для действия такой группы на $2^{\mathbb{Q}^n}$ стандартного симплекса можно рассматривать {\it равнобедренный} стандартный симплекс фиксированного объема $$(x_0,y_0,\dots ,z_0), \quad (x_0+a,y_0,\dots ,z_0), \quad (x_0,y_0+b,\dots ,z_0),\quad \dots , \quad (x_0,y_0,\dots ,z_0+c),$$ где $x_0,y_0,\dots ,z_0,a,b,\dots ,c\in\mathbb{Q}$, $a,b,\dots ,c>0$, $b=\dots =c$, $\dfrac{1}{n}ab\dots c=V.$


\newpage


\end{document}